\documentclass[11pt]{article}

\usepackage{amsfonts}
\usepackage{amsmath}
\usepackage{amssymb}
\usepackage{amsthm}
\usepackage{geometry}
\usepackage[latin1]{inputenc}
\usepackage[english]{babel}

\usepackage{bm}
\usepackage{natbib}

\usepackage{setspace}

\theoremstyle{plain} \newtheorem{lemma}{Lemma}
\theoremstyle{plain} 
\theoremstyle{remark} \newtheorem{remark}{Remark}

\begin{document}
\title{On two simple tests for normality with high power}

\author{ \bf{M\aa ns Thulin}$^{1}$}
\date{First version: August 31, 2010. Updated: \today}

\maketitle
\footnotetext[1]{Department of Mathematics, Uppsala University, P.O.Box 480, 751 06 Uppsala, Sweden.\\Phone: +46(0)184713389; Fax: +46(0)184713201; E-mail: thulin@math.uu.se}


\begin{abstract}
\noindent The test statistics of two powerful tests for normality \citep{lm1,mud2} are estimators of the correlation coefficient between certain sample moments. We derive new versions of the test statistics that are functions of the sample skewness and sample kurtosis. This sheds some light on the nature of these tests and leads to easier computations.
\end{abstract}


\section{Introduction}
The assumption of normality is the basis of many of the most common statistical methods. Tests for normality, used to assess the normality assumption, is therefore a widely studied field. \citet{tho1} provides an overview. Some popular tests for normality are based on the sample skewness and sample kurtosis, described in Section \ref{motivation} below. Others use characterizations of the normal distribution.

One such well-known characterization is that the sample mean $\bar{X}$ and sample variance $S^2$ are independent if and only if the underlying population is normal. Similarly, $\bar{X}$ and $n^{-1}\sum_{i=1}^n(X_i-\bar{X})^3$ are independent if and only if $X$ is normal; see \citet{klr1}, Sections 4.2 and 4.7.

\citet{lm1} proposed a test based on the independence of $\bar{X}$ and $S^2$. They noted that it is difficult to test the independence of $\bar{X}$ and $S^2$ but that the correlation coefficient between the two is possible to estimate. They used a jackknife procedure to estimate $\rho(\bar{X}, S^2)$, and used this for a test for normality against asymmetric alternatives. The test has been modified, generalized and discussed in \citet{br1}, \citet{mud1} and \citet{wm1}. In \citet{mud2} a test based on the independence of $\bar{X}$ and $n^{-1}\sum_{i=1}^n(X_i-\bar{X})^3$ was proposed, constructed using the same jackknife procedure. The authors named the tests the $Z_2$ test and $Z_3$ test.

In this paper we show that it is possible to replace the jackknife estimators used by Lin and Mudholkar and Mudholkar et~al. by estimators that are smooth functions of the sample skewness and sample kurtosis. In Section 2 we describe the $Z_2$ and $Z_3$ tests and derive the new estimators. In Section 3 we present some simulation results that indicate that the tests have very good power properties. Throughout the text we use the notation $\mu_k=E(X-\mu)^k$ to denote central moments.

\section{The $Z$ tests and correlation coefficients}

\subsection{The $Z_2$ and $Z_3$ tests}\label{oldtests}
Lin and Mudholkar used the $n$ jackknife replications $(\bar{X}_{-i},S^2_{-i})$, where
\[
\bar{X}_{-i}=\frac{1}{n-1}\sum_{j\neq i}X_j,\qquad S^2_{-i}=\frac{1}{n-2}\sum_{j\neq i}(X_j-\bar{X}_{-i})^2,
\]
to study the dependence between $\bar{X}$ and $S^2$. They applied the cube-root transformation $Y_i=(S^2_{-i})^{1/3}$ and concluded that the sample correlation coefficient $r(\bar{X}_{-i},Y_i)$ equals the sample correlation coefficient
\[
r_2=r(X_i,Y_i)=\frac{\sum_{i=1}^n (X_i-\bar{X})(Y_i-\bar{Y})}{\sqrt{\sum_{i=1}^n (X_i-\bar{X})^2 \sum_{i=1}^n(Y_i-\bar{Y})^2}}.
\]
Finally, they used Fisher's $z$-transform to obtain the test statistic
\[
Z_2=\frac{1}{2}\log \Big{(}\frac{1+r_2}{1-r_2} \Big{)}
\]
and used this for their test for normality. The test is sensitive to departures from normality in the form of skewness. If the sign of the skewness of the alternative is known, a one-tailed test can be used. If it is unknown, a two-tailed test is used. The latter will be refered to as the $|Z_2|$ test.

\citet{mud2} used the same jackknife approach to construct another test for normality. This time they considered the mean $\bar{X}$ and the third central sample moment $\hat{\mu}_3=n^{-1}\sum_{i=1}^n(X_i-\bar{X})^3$. Letting
\[
\bar{X}_{-i}=\frac{1}{n-1}\sum_{j\neq i}X_j,\qquad \hat{\mu}_{3,-i}=\frac{1}{n-1}\sum_{j\neq i}(X_j-\bar{X}_{-i})^3=Y_i,
\]
they used the sample correlation coefficient
\[
r_3=r(X_i,Y_i)=\frac{\sum_{i=1}^n (X_i-\bar{X})Y_i}{\sqrt{\sum_{i=1}^n (X_i-\bar{X})^2 \sum_{i=1}^n(Y_i-\bar{Y})^2}}
\]
in the same manner as in the above test, obtaining the test statistic
\[
Z_3=\frac{1}{2}\log \Big{(}\frac{1+r_3}{1-r_3} \Big{)}.
\]
The simulation results in \citet{mud2} indicate that both tests have high power against some interesting alternatives.


\subsection{Explicit expressions}\label{motivation}
Next, we derive explicit expressions for the correlation coefficients $\rho(\bar{X}, S^2)$ and $\rho(\bar{X},\hat{\mu}_3)$, which enables us to estimate the correlation coefficients using sample moments. The estimators considered in the correlations will be the unbiased estimators $S^2$ and $\hat{\mu}_{3}=\frac{n}{(n-1)(n-2)}\sum_{i=1}^n(X_i-\bar{X})^3$.

The formulae obtained are somewhat easier to express using standardized cumulants. Let $\varkappa_1, \varkappa_2, \ldots$ denote the cumulants of $X$. The $k$th standardized cumulant of $X$ is $\frac{\varkappa_k}{\varkappa_2^{k/2}}$. We are particularly interested in
\[\begin{split}
\gamma=\frac{\varkappa_3}{\varkappa_2^{3/2}}=\frac{\mu_3}{\sigma^ 3},\qquad\kappa=\frac{\varkappa_4}{\varkappa_2^{2}}=\frac{\mu_4}{\sigma^4}-3,\qquad\lambda=\frac{\varkappa_6}{\varkappa_2^{3}}=\frac{\mu_6}{\sigma^6}-15\kappa-10\gamma^2-15.
\end{split}\]

$\gamma$ is the skewness of $X$ and $\kappa$ is the (excess) kurtosis of $X$. All cumulants are 0 for the normal distribution. 

\begin{lemma}\label{korrlemma}
Suppose that $X_1, \ldots, X_n$ are independent and identically distributed random variables. Denote their mean $\mu$, variance $\sigma^ 2$, skewness $\gamma$ and kurtosis $\kappa$. Let $\bar{X}=\frac{1}{n}\sum_{i=1}^nX_i$, $S^2=\frac{1}{n-1}\sum_{i=1}^n(X_i-\bar{X})^2$ and $\hat{\mu}_{3}=\frac{n}{(n-1)(n-2)}\sum_{i=1}^n(X_i-\bar{X})^3$. Then the following results hold.
\begin{enumerate}
\item[(i)] If $EX^4<\infty$ and $n\geq 2$,
\begin{equation}\label{korrelation}
\rho_2=\rho(\bar{X}, S^2)=\frac{\mu_3}{\sigma^ 3\sqrt{\frac{\mu_4}{\sigma^4}-\frac{n-3}{n-1}}}=\frac{\gamma}{\sqrt{\kappa+3-\frac{n-3}{n-1}}}.
\end{equation}
\item[(ii)] If $EX^6<\infty$ and $n\geq 3$,
\begin{equation}\label{korrelation2}\begin{split}
\rho_3=\rho(\bar{X},\hat{\mu}_3)&=\frac{\mu_4-3\sigma^4}{\sigma^4\sqrt{\frac{\mu_6}{\sigma^6}-3\frac{(2n-5)}{n-1}\frac{\mu_4}{\sigma^4}-\frac{(n-10)}{(n-1)}\frac{\mu_3^ 2}{\sigma^6}+\frac{(9n^2-36n+60)}{(n-1)(n-2)}}}\\
&=\frac{\kappa}{\sqrt{\lambda+9\frac{n}{n-1}(\kappa+\gamma^2)+\frac{6n^2}{(n-1)(n-2)}}},
\end{split}\end{equation}
where $\lambda$ is the sixth standardized cumulant of $X$.
\end{enumerate}
\end{lemma}
\begin{proof}
The proof amounts to calculating the moments involved:
\begin{enumerate}
\item[(i)] It is well-known that $\mbox{var}(\bar{X})=\sigma^2/n$ and from the results of Section 27.4 of \citet{cr1} it follows that \[
\mbox{var}(S^2)=\mbox{var}\Big{(}\frac{n}{n-1}\frac{1}{n}\sum_{i=1}^n(X_i-\bar{X})^2)\Big{)}=\frac{1}{n}\mu_4-\frac{n-3}{n(n-1)}\sigma^4
\]
and that
\[
\mbox{cov}(\bar{X}, S^2)=\frac{n}{n-1}\frac{n-1}{n^2}\mu_3=\frac{1}{n}\mu_3.
\]
The result follows from the above moments and the fact that $\mu_4=\sigma^4(\kappa+3)$.

\item[(ii)] From \citet{fi1} we have
\[\begin{split}
\mbox{var}(\hat{\mu}_{3})=\frac{1}{n}\lambda\sigma^6+\frac{9(\kappa+\gamma^6)\sigma^6}{n-1}+\frac{6n\sigma^6}{(n-1)(n-2)}.
\end{split}\]
Furthermore,
\[
\mbox{cov}(\bar{X},\hat{\mu}_{3})=E((\bar{X}-\mu)\hat{\mu}_{3})-E((\bar{X}-\mu)\mu_3)=E((\bar{X}-\mu)\hat{\mu}_{3}),
\]
but this expression does not depend on $\mu$, so we can study the case where $\mu=0$ without loss of generality. Then
\[\begin{split}
&\mbox{cov}(\bar{X},\hat{\mu}_{3})=E(\bar{X}\hat{\mu}_{3})=\frac{n^2}{(n-1)(n-2)}E\Big{(}\bar{X}\frac{1}{n}\Big{(}\sum_iX_i^3-3\bar{X}\sum_iX_i^2+3\bar{X}^2\sum_iX_i-\bar{X}^3\Big{)}\Big{)}\\
&=\frac{n^2}{(n-1)(n-2)}\Big{(} E(\bar{X}\frac{1}{n}\sum_iX_i^3) -3E(\bar{X}^2\frac{1}{n}\sum_iX_i^2)+2E(\bar{X}^4)  \Big{)}.
\end{split}\]
The three expectations above are all found in Sections 27.4 and 27.5 of \citet{cr1}. Inserting their values, routine calculations give that the above expressions reduces to
\[\begin{split}
\mbox{cov}(\bar{X},\hat{\mu}_{3})=\frac{\mu_4-3\sigma^4}{n}.
\end{split}\]
Thus
\[\begin{split}
&\rho(\bar{X},\hat{\mu}_3)=\frac{\frac{\mu_4-3\sigma^4}{n}}{\sqrt{\frac{1}{n}\sigma^2}\sqrt{\frac{1}{n}\lambda\sigma^6+\frac{9(\kappa+\gamma^6)\sigma^6}{n-1}+\frac{6n\sigma^6}{(n-1)(n-2)}}}\\
&=\frac{\kappa}{\sqrt{\lambda+9\frac{n}{n-1}(\kappa+\gamma^2)+\frac{6n^2}{(n-1)(n-2)}}}=\frac{\mu_4-3\sigma^4}{\sigma^4\sqrt{\frac{\mu_6}{\sigma^6}-3\frac{(2n-5)}{n-1}\frac{\mu_4}{\sigma^4}-\frac{(n-10)}{(n-1)}\frac{\mu_3^ 2}{\sigma^6}+\frac{(9n^2-36n+60)}{(n-1)(n-2)}}}.
\end{split}\]
\end{enumerate}
\end{proof}

\begin{remark}
\citet{ks1}, Section 31.3, present the asymptotic result that $\rho(\bar{X}, S^2)\rightarrow\frac{\gamma}{\sqrt{\kappa+2}}$.
\end{remark}

The following little-known lemma, relating the standardized cumulants to each other, tells us what the possible values of $(\gamma,\kappa,\lambda)$ are. This allows us to study $\rho_2$ and $\rho_3$ as functions of the standardized cumulants. 

\begin{lemma}\label{otherslemma}
Let $X,X_1, X_2, \ldots$ be independent and identically distributed random variables that satisfy the conditions in Lemma \ref{korrlemma}. Then
\begin{enumerate}
\item[(i)] $\gamma^2\leq\kappa+2$, with equality if and only if $X$ has a two-point distribution.
\item[(ii)] $\kappa^2\leq \lambda+9(\kappa+\gamma^2)+6$, with equality if $X$ has a two-point distribution.
\end{enumerate}
\end{lemma}
\noindent
The inequality in (i) was first shown by \citet{dm1}. The entire statement was later shown by \citet{rs1}. (ii) follows from expression (13) in \citet{dm1} when $\gamma^2<\kappa+2$. It is readily verified that equality holds for two-point distributions. We have not found an $X$ distributed on more than two points for which equality holds in (ii), and conjecture that $\kappa^2= \lambda+9(\kappa+\gamma^2)+6$ only if $X$ has a two-point distribution.


\begin{remark}
Incidentally, Lemma \ref{otherslemma} can be used to verify that $\rho_2$ and $\rho_3$ are bounded by $-1$ and $1$. Since $\kappa+2\geq \gamma^2$ we have $|\gamma|/\sqrt{\kappa+3-\frac{n-3}{n-1}}<|\gamma|/\sqrt{\kappa+2}\leq 1$ and hence $|\rho_2| < 1$, as expected. We see that the correlation coefficient never equals $\pm 1$. Looking at $\rho_3$ we similarly get $|\rho_3|< 1$ since $\kappa^2\leq \lambda+9(\kappa+\gamma^2)+6$. Conversely, the fact that $\rho_2$ and $\rho_3$ must be bounded by $-1$ and $1$ can be used as a partial proof of Lemma \ref{otherslemma}.
\end{remark}


\subsection{Test statistics}\label{estimators}
The moment estimators of the two correlation coefficients are now obtained by replacing the moments by their sample counterparts
\[\begin{split}
\hat{\gamma}&=\frac{\frac{1}{n}\sum_{i=1}^n(x_i-\bar{x})^3}{\Big{(}\frac{1}{n}\sum_{i=1}^n(x_i-\bar{x})^2\Big{)}^{3/2}},\qquad\hat{\kappa}=\frac{\frac{1}{n}\sum_{i=1}^n(x_i-\bar{x})^4}{\Big{(}\frac{1}{n}\sum_{i=1}^n(x_i-\bar{x})^2\Big{)}^{2}}-3,\\
\hat{\lambda}&=\frac{\frac{1}{n}\sum_{i=1}^n(x_i-\bar{x})^6}{\Big{(}\frac{1}{n}\sum_{i=1}^n(x_i-\bar{x})^2\Big{)}^{3}}-15\hat{\kappa}-10\hat{\gamma}^2-15.
\end{split}\]
The estimators are thus defined as
\begin{eqnarray}
&\label{estimator1}Z_2'=\frac{\hat{\gamma}}{\sqrt{\hat{\kappa}+3-\frac{n-3}{n-1}}},\\
&\label{estimator2}Z_3'=\frac{\hat{\kappa}}{\sqrt{\hat{\lambda}+9\frac{n}{n-1}(\hat{\kappa}+\hat{\gamma}^2)+\frac{6n^2}{(n-1)(n-2)}}}.
\end{eqnarray}
From the above equations and Lemma \ref{otherslemma} it is clear that $Z_2'$ in fact is nothing but a smooth function of the sample skewness and sample kurtosis, relating the size of these two quantities. Likewise, $Z_3'$ is a smooth function of the sample kurtosis, skewness and sixth cumulant.

The estimators are clearly scale and location invariant, i.e. independent of $\mu$ and $\sigma$, as $\hat{\gamma}$, $\hat{\kappa}$ and $\hat{\lambda}$ all share that property. They are therefore suitable as test statistics for tests for normality. Furthermore, it follows from the Cram\'{e}r-Slutsky lemma that they are consistent whenever the necessary moments exist. Asymptotic normality can be shown as well, but the convergence is slow and we therefore prefer to obtain the null distribution using Monte Carlo simulation.

\subsection{Testing}\label{testing}
From Lemma \ref{korrlemma} we conclude that $\rho_2$ is large when the underlying distribution has high skewness, and that high kurtosis brings the correlation coefficient closer to 0. When using $Z_2'$ as test statistic for a normality test, we should thus reject the null hypothesis of normality if, when the alternative distribution has positive skewness, $Z_2'$ is unusually large, or if, when the alternative distribution has negative skewness, $Z_2'$ is negative and unusually large. If the sign of the skewness of the alternative is unknown, $|Z_2'|$ can be studied instead.

Similarly, the hypothesis of normality should be rejected if $Z_3'$ is far from 0. If the sign of the kurtosis of the alternative is known, a one-tailed test should be used.

From the consistency of the estimators (\ref{estimator1}) and (\ref{estimator2}) it is clear that $Z_2'$ test is consistent against alternatives with $\gamma\neq 0$ and that the $Z_3'$ test is consistent against alternatives with $\kappa\neq 0$.

An R implementation of the test is found in the \texttt{cornormtest} package, available from the author.

\section{A simulation power study}
\subsection{Study}

To evalute the performance of the tests a simulation power study was performed, where the $|Z_2'|$, $Z_2'$ and $Z_3'$ test were compared to the $|Z_2|$, $Z_2$ and $Z_3$ tests, one-tailed versions of the sample moment tests $\sqrt{b_1}=\hat{\gamma}$ and $b_2=\hat{\kappa}$, the Shapiro--Wilk test $W$ \citep{sha1}, Vasicek's test $K$ \citep{va1} and the Jarque-Bera test $LM$ \citep{jb1}. The latter test has performed poorly in previous comparisons of power, but is nevertheless popular in econometrics. It is of some interest to us since it is based on the sample skewness and kurtosis; the test statistic is $LM=n(\frac{1}{6}\hat{\gamma}^2+\frac{1}{24}\hat{\kappa}^2)$.

The tests were studied for $\chi^2$, Weibull, lognormal, beta, Student's $t$, Laplace, logistic and normal mixture alternatives and were thus compared for both symmetric and asymmetric distributions as well as short-tailed and long-tailed ones. The skewness, kurtosis and limit correlation coefficients of the alternatives are given in Table \ref{tab0}. To estimate their powers against the various alternative distributions at the significance level $\alpha=0\cdot 05$, the tests were applied to 1,000,000 simulated random samples of size $n=20$ and $n=50$ from each distribution. 

\begin{table}[ht]
\begin{center}
\footnotesize{
\caption{Skewness, kurtosis and correlation coefficients for distributions in the study}\label{tab0}

\begin{tabular}{|l| c c c c| }
\hline
Distribution &  $\mathbf{\gamma}$& $\mathbf{\kappa}$ &   $\mathbf{\lim_{n\rightarrow\infty}\rho_2}$ & $\mathbf{\lim_{n\rightarrow\infty}\rho_3}$\\ \hline

Normal & 0 & 0 & 0 & 0\\

$\chi^2(1)$ & 2$\cdot$82 & 12 & 0$\cdot$75 & 0$\cdot$46 \\

Exponential & 2 & 6 & 0$\cdot$71 & 0$\cdot$41 \\

$\chi^2(4)$ & 1$\cdot$41 & 3 & 0$\cdot$63 & 0$\cdot$33 \\

Weib(1/2,1) & 6$\cdot$62 & 84$\cdot$72   & 0$\cdot$71 & 0$\cdot$36 \\

Weib(2,1) & 0$\cdot$63 & 0$\cdot$25  & 0$\cdot$42 & 0$\cdot$07 \\

LN($\sigma=1/4$)  & 0$\cdot$78 & 1$\cdot$10  & 0$\cdot$32 & 0$\cdot$21 \\

LN($\sigma=1/2$)  & 1$\cdot$75 & 5$\cdot$90  & 0$\cdot$53 & 0$\cdot$33 \\

Beta(1/2,1/2)  & 0 & -1$\cdot$5 & 0 & -0$\cdot$95 \\

Uniform & 0 & -1$\cdot$2  & 0 & -0$\cdot$84 \\

Beta(2,2) & 0 & -0$\cdot$86   & 0 & -0$\cdot$59 \\

Beta(3,3) & 0 & -2/3   & 0 & -0$\cdot$43 \\

Beta(1,2) & 0$\cdot$57 & -0$\cdot$6  & 0$\cdot$48 & -0$\cdot$34 \\

Beta(2,3) & 0$\cdot$29&-0$\cdot$64 & 0$\cdot$25 & -0$\cdot$39 \\

Cauchy & - & -  & - & - \\

t(2) & - & -  & - & - \\

t(3) & - & - & - & - \\

t(4) & 0 & -  & - & - \\

t(5) & 0 & 6  & 0 & - \\

t(6) & 0 & 3     & 0 & - \\

Laplace  & 0 & 3    & 0 & 0$\cdot$38 \\

Logistic  & 0 & 1$\cdot$2   & 0 & 0$\cdot$25 \\

\tiny{$\frac{1}{2}$N(0,1)+$\frac{1}{2}$N(1,1)} & 0 & -0$\cdot$08  & 0 & -0$\cdot$03 \\

\tiny{$\frac{1}{2}$N(0,1)+$\frac{1}{2}$N(4,1)} & 0 & -1$\cdot$28 & 0 & -0$\cdot$78 \\ 

\tiny{$\frac{9}{10}$N(0,1)+$\frac{1}{10}$N(4,1)} & 1$\cdot$2 & 1$\cdot$78 & 0$\cdot$62 & 0$\cdot$44 \\
\hline
\end{tabular}
}
\end{center}
\end{table}

\begin{table}[ht]
\begin{center}
\footnotesize{
\caption{Power of tests for normality against some alternatives, $\alpha=0\cdot 05$, $n=20$}\label{tab2}
\begin{tabular}{|l| c c c c c c c c c c c| }
\hline
$\mathbf{n=20}$ &  $\mathbf{W}$&     $\mathbf{K}$ & $\mathbf{LM}$ & $\mathbf{\sqrt{b_1}}$ &$\mathbf{b_2}$ & $\mathbf{|Z_2|}$ &$\mathbf{Z_2}$&$\mathbf{Z_3}$ &$\mathbf{|Z_2'|}$ & $ \mathbf{Z_2'}$ & $\mathbf{Z_3'}$\\ \hline

$\chi^2(1)$ &  0$\cdot$98 & 0$\cdot$99 & 0$\cdot$72 & 0$\cdot$95& 0$\cdot$61 & 0$\cdot$97 & 0$\cdot$98 & 0$\cdot$64 &  0$\cdot$97 & 0$\cdot$98 & 0$\cdot$62 \\

Exponential  & 0$\cdot$84 & 0$\cdot$84 & 0$\cdot$48 & 0$\cdot$81& 0$\cdot$43 & 0$\cdot$82 & 0$\cdot$89 & 0$\cdot$42 &  0$\cdot$82 & 0$\cdot$89 & 0$\cdot$42 \\

$\chi^2(4)$  & 0$\cdot$53 & 0$\cdot$45 & 0$\cdot$29 & 0$\cdot$60& 0$\cdot$27 & 0$\cdot$55 & 0$\cdot$68 & 0$\cdot$26 &  0$\cdot$57 & 0$\cdot$68 & 0$\cdot$26\\  

Weib(1/2,1)   & 1$\cdot$00 & 1$\cdot$00 & 0$\cdot$90 & 0$\cdot$99& 0$\cdot$83 & 1$\cdot$00 & 1$\cdot$00 & 0$\cdot$85 &  1$\cdot$00 & 1$\cdot$00 & 0$\cdot$84 \\

Weib(2,1)   & 0$\cdot$15 & 0$\cdot$13 & 0$\cdot$07& 0$\cdot$23 & 0$\cdot$08 & 0$\cdot$17 & 0$\cdot$27 & 0$\cdot$07 &  0$\cdot$17 & 0$\cdot$27 & 0$\cdot$08 \\

LN($\sigma=1/4$)   & 0$\cdot$19 & 0$\cdot$12 & 0$\cdot$12& 0$\cdot$29 & 0$\cdot$14 & 0$\cdot$20 & 0$\cdot$30 & 0$\cdot$13 &  0$\cdot$21 & 0$\cdot$31 & 0$\cdot$13 \\

LN($\sigma=1/2$)  & 0$\cdot$52 & 0$\cdot$40 & 0$\cdot$33& 0$\cdot$62 & 0$\cdot$33 & 0$\cdot$56 & 0$\cdot$67 & 0$\cdot$31 &  0$\cdot$57 & 0$\cdot$67 & 0$\cdot$32 \\

Beta(1/2,1/2)   & 0$\cdot$72 & 0$\cdot$92 & 0$\cdot$00 & 0$\cdot$02& 0$\cdot$77 & 0$\cdot$13 & 0$\cdot$10 & 0$\cdot$82 &  0$\cdot$12 & 0$\cdot$09 & 0$\cdot$78 \\

Uniform   & 0$\cdot$20 & 0$\cdot$42 & 0$\cdot$00 & 0$\cdot$01 & 0$\cdot$44& 0$\cdot$04 & 0$\cdot$04 & 0$\cdot$51 &  0$\cdot$04 & 0$\cdot$04 & 0$\cdot$46 \\

Beta(2,2)   & 0$\cdot$05 & 0$\cdot$13 & 0$\cdot$00 & 0$\cdot$01 & 0$\cdot$18& 0$\cdot$02 & 0$\cdot$03 & 0$\cdot$21 &  0$\cdot$02 & 0$\cdot$03 & 0$\cdot$18 \\

Beta(3,3)   & 0$\cdot$04 & 0$\cdot$09 & 0$\cdot$00 & 0$\cdot$02 & 0$\cdot$11& 0$\cdot$02 & 0$\cdot$03 & 0$\cdot$13 &  0$\cdot$03 & 0$\cdot$03 & 0$\cdot$11 \\

Beta(1,2)   & 0$\cdot$30 & 0$\cdot$43 & 0$\cdot$03 & 0$\cdot$22 & 0$\cdot$17& 0$\cdot$24 & 0$\cdot$37 & 0$\cdot$18 &  0$\cdot$24 & 0$\cdot$37 & 0$\cdot$16 \\

Beta(2,3) &  0$\cdot$07 & 0$\cdot$12 & 0$\cdot$02 & 0$\cdot$07 & 0$\cdot$13& 0$\cdot$06 & 0$\cdot$11 & 0$\cdot$15 &  0$\cdot$06 & 0$\cdot$11 & 0$\cdot$13 \\

Cauchy   & 0$\cdot$87 & 0$\cdot$74 & 0$\cdot$82 & 0$\cdot$41 & 0$\cdot$88 & 0$\cdot$70 & 0$\cdot$37 & 0$\cdot$90 & 0$\cdot$70 & 0$\cdot$37 & 0$\cdot$89 \\

t(2)   & 0$\cdot$53 & 0$\cdot$31 & 0$\cdot$49 & 0$\cdot$29 & 0$\cdot$59 & 0$\cdot$43 & 0$\cdot$25 & 0$\cdot$61 & 0$\cdot$43 & 0$\cdot$25 & 0$\cdot$61 \\

t(3)   & 0$\cdot$34 & 0$\cdot$16 & 0$\cdot$31 & 0$\cdot$21 &0$\cdot$40& 0$\cdot$29 & 0$\cdot$19 & 0$\cdot$42 &  0$\cdot$29 & 0$\cdot$18 & 0$\cdot$42\\

t(4)   & 0$\cdot$24 & 0$\cdot$10 & 0$\cdot$22 & 0$\cdot$17 & 0$\cdot$30& 0$\cdot$21 & 0$\cdot$15 & 0$\cdot$31 & 0$\cdot$21 & 0$\cdot$15 & 0$\cdot$31 \\

t(5)  & 0$\cdot$19 & 0$\cdot$07 & 0$\cdot$17 & 0$\cdot$14 & 0$\cdot$24& 0$\cdot$17 & 0$\cdot$12 & 0$\cdot$25 & 0$\cdot$17 & 0$\cdot$12 & 0$\cdot$25 \\

t(6)  & 0$\cdot$15 & 0$\cdot$06 & 0$\cdot$13 & 0$\cdot$13 & 0$\cdot$20& 0$\cdot$14 & 0$\cdot$11 & 0$\cdot$20 & 0$\cdot$14 & 0$\cdot$11 & 0$\cdot$21 \\

Laplace   & 0$\cdot$26 & 0$\cdot$09 & 0$\cdot$22 & 0$\cdot$17 & 0$\cdot$33& 0$\cdot$20 & 0$\cdot$14 & 0$\cdot$36 & 0$\cdot$20& 0$\cdot$14 & 0$\cdot$35 \\

Logistic   & 0$\cdot$12 & 0$\cdot$05 & 0$\cdot$10 & 0$\cdot$10& 0$\cdot$16& 0$\cdot$11 & 0$\cdot$09 & 0$\cdot$16 & 0$\cdot$11 & 0$\cdot$09 & 0$\cdot$16 \\

\tiny{$\frac{1}{2}$N(0,1)+$\frac{1}{2}$N(1,1)}  & 0$\cdot$05 & 0$\cdot$05 & 0$\cdot$02& 0$\cdot$05 & 0$\cdot$05 & 0$\cdot$05 & 0$\cdot$05 & 0$\cdot$06 &  0$\cdot$04 & 0$\cdot$05 & 0$\cdot$05 \\

\tiny{$\frac{1}{2}$N(0,1)+$\frac{1}{2}$N(4,1)}  & 0$\cdot$40 & 0$\cdot$55 & 0$\cdot$00& 0$\cdot$02 & 0$\cdot$61 & 0$\cdot$09 & 0$\cdot$08 & 0$\cdot$58 &  0$\cdot$09 & 0$\cdot$08 & 0$\cdot$55 \\

\tiny{$\frac{9}{10}$N(0,1)+$\frac{1}{10}$N(4,1)}  & 0$\cdot$53 & 0$\cdot$27& 0$\cdot$35 & 0$\cdot$65 & 0$\cdot$38 & 0$\cdot$53 & 0$\cdot$64 & 0$\cdot$42 &  0$\cdot$53 & 0$\cdot$64 & 0$\cdot$40 \\
\hline
\end{tabular}
}
\end{center}
\end{table}

\begin{table}[ht]
\begin{center}
\footnotesize{
\caption{Power of tests for normality against some alternatives, $\alpha=0\cdot 05$, $n=50$}\label{tab3}

\begin{tabular}{|l| c c c c c c c c c c c| }
\hline
$\mathbf{n=50}$ &  $\mathbf{W}$&     $\mathbf{K}$ & $\mathbf{LM}$ & $\mathbf{\sqrt{b_1}}$ &$\mathbf{b_2}$ & $\mathbf{|Z_2|}$ &$\mathbf{Z_2}$&$\mathbf{Z_3}$ &$\mathbf{|Z_2'|}$ & $ \mathbf{Z_2'}$ & $\mathbf{Z_3'}$\\ \hline

$\chi^2(1)$ &  1$\cdot$00 & 1$\cdot$00 & 1$\cdot$00 & 1$\cdot$00& 0$\cdot$91 & 1$\cdot$00 & 1$\cdot$00 & 0$\cdot$92 & 1$\cdot$00 & 1$\cdot$00 & 0$\cdot$92 \\

Exponential  & 1$\cdot$00 & 1$\cdot$00 & 0$\cdot$95 & 1$\cdot$00 & 0$\cdot$73& 1$\cdot$00 & 1$\cdot$00 & 0$\cdot$73 &  1$\cdot$00 & 1$\cdot$00 & 0$\cdot$73 \\

$\chi^2(4)$  & 0$\cdot$95 & 0$\cdot$91 & 0$\cdot$76 & 0$\cdot$95 & 0$\cdot$50& 0$\cdot$95 & 0$\cdot$97 & 0$\cdot$48 &  0$\cdot$95 & 0$\cdot$98 & 0$\cdot$49 \\  

Weib(1/2,1)   & 1$\cdot$00 & 1$\cdot$00 & 1$\cdot$00 & 1$\cdot$00& 0$\cdot$99 & 1$\cdot$00 & 1$\cdot$00 & 0$\cdot$99 &  1$\cdot$00 & 1$\cdot$00 & 0$\cdot$99 \\

Weib(2,1)   & 0$\cdot$41 & 0$\cdot$32 & 0$\cdot$21 & 0$\cdot$52& 0$\cdot$12 & 0$\cdot$45 & 0$\cdot$58 & 0$\cdot$10 &   0$\cdot$46 & 0$\cdot$60 & 0$\cdot$10 \\

LN($\sigma=1/4$)   & 0$\cdot$44 & 0$\cdot$25 & 0$\cdot$34& 0$\cdot$59 & 0$\cdot$24 & 0$\cdot$49 & 0$\cdot$61 & 0$\cdot$22 &  0$\cdot$51 & 0$\cdot$62 & 0$\cdot$23 \\

LN($\sigma=1/2$)  & 0$\cdot$92 & 0$\cdot$83 & 0$\cdot$80& 0$\cdot$95 & 0$\cdot$60 & 0$\cdot$94 & 0$\cdot$97 & 0$\cdot$59 &   0$\cdot$94 & 0$\cdot$97 & 0$\cdot$59 \\

Beta(1/2,1/2)   & 1$\cdot$00 & 1$\cdot$00 & 0$\cdot$03 & 0$\cdot$01& 1$\cdot$00 & 0$\cdot$14 & 0$\cdot$10 & 1$\cdot$00 &  0$\cdot$13 & 0$\cdot$10 & 1$\cdot$00 \\

Uniform   & 0$\cdot$75 & 0$\cdot$92 & 0$\cdot$00 & 0$\cdot$01& 0$\cdot$94 & 0$\cdot$04 & 0$\cdot$04 & 0$\cdot$96 &  0$\cdot$04 & 0$\cdot$04 & 0$\cdot$96 \\

Beta(2,2)   & 0$\cdot$15 & 0$\cdot$31 & 0$\cdot$00 & 0$\cdot$01 & 0$\cdot$52& 0$\cdot$02 & 0$\cdot$02 & 0$\cdot$55 &  0$\cdot$02 & 0$\cdot$02 & 0$\cdot$55 \\

Beta(3,3)   & 0$\cdot$07 & 0$\cdot$15 & 0$\cdot$00& 0$\cdot$01 & 0$\cdot$28& 0$\cdot$02 & 0$\cdot$02 & 0$\cdot$30 &  0$\cdot$02 & 0$\cdot$03 & 0$\cdot$30 \\

Beta(1,2)   & 0$\cdot$84 & 0$\cdot$91 & 0$\cdot$11 & 0$\cdot$52 & 0$\cdot$31& 0$\cdot$61 & 0$\cdot$74 & 0$\cdot$28 &  0$\cdot$63 & 0$\cdot$75 & 0$\cdot$28 \\

Beta(2,3) &  0$\cdot$20 & 0$\cdot$29 & 0$\cdot$01 & 0$\cdot$13 & 0$\cdot$30& 0$\cdot$12 & 0$\cdot$21 & 0$\cdot$31 &  0$\cdot$13 & 0$\cdot$23 & 0$\cdot$31 \\

Cauchy  & 1$\cdot$00 & 0$\cdot$99 & 0$\cdot$99 & 0$\cdot$46 & 1$\cdot$00 & 0$\cdot$82 & 0$\cdot$42 & 1$\cdot$00 &  0$\cdot$81& 0$\cdot$42 & 1$\cdot$00 \\

t(2)   & 0$\cdot$86 & 0$\cdot$68 & 0$\cdot$87 & 0$\cdot$38 & 0$\cdot$90 & 0$\cdot$58 & 0$\cdot$32 & 0$\cdot$92 &  0$\cdot$58& 0$\cdot$32 & 0$\cdot$92 \\

t(3)   & 0$\cdot$64 & 0$\cdot$37 & 0$\cdot$67 & 0$\cdot$30 & 0$\cdot$73 & 0$\cdot$41 & 0$\cdot$24 & 0$\cdot$75 &  0$\cdot$41& 0$\cdot$25 & 0$\cdot$74 \\

t(4)   & 0$\cdot$47 & 0$\cdot$21 & 0$\cdot$50 & 0$\cdot$24 & 0$\cdot$57 & 0$\cdot$30 & 0$\cdot$19 & 0$\cdot$59 &  0$\cdot$31& 0$\cdot$19 & 0$\cdot$59 \\

t(5)   & 0$\cdot$36 & 0$\cdot$13 & 0$\cdot$39 & 0$\cdot$20 & 0$\cdot$46 & 0$\cdot$24 & 0$\cdot$16 & 0$\cdot$47 &  0$\cdot$24& 0$\cdot$16 & 0$\cdot$47 \\

t(6)   & 0$\cdot$28 & 0$\cdot$10 & 0$\cdot$32 & 0$\cdot$17 & 0$\cdot$38 & 0$\cdot$19 & 0$\cdot$14 & 0$\cdot$39 &  0$\cdot$20& 0$\cdot$14 & 0$\cdot$39 \\

Laplace   & 0$\cdot$52 & 0$\cdot$26 & 0$\cdot$51 & 0$\cdot$22 & 0$\cdot$61& 0$\cdot$24 & 0$\cdot$16 & 0$\cdot$67 &  0$\cdot$25 & 0$\cdot$17 & 0$\cdot$66 \\

Logistic   & 0$\cdot$20 & 0$\cdot$06 & 0$\cdot$22 & 0$\cdot$14& 0$\cdot$28 & 0$\cdot$14 & 0$\cdot$11 & 0$\cdot$29 &  0$\cdot$13 & 0$\cdot$11 & 0$\cdot$29 \\

\tiny{$\frac{1}{2}$N(0,1)+$\frac{1}{2}$N(1,1)}  & 0$\cdot$05 & 0$\cdot$05 & 0$\cdot$03& 0$\cdot$05 & 0$\cdot$06 & 0$\cdot$04 & 0$\cdot$04 & 0$\cdot$06 &  0$\cdot$04 & 0$\cdot$05 & 0$\cdot$06 \\

\tiny{$\frac{1}{2}$N(0,1)+$\frac{1}{2}$N(4,1)}  & 0$\cdot$90 & 0$\cdot$92 & 0$\cdot$00& 0$\cdot$01 & 0$\cdot$96 & 0$\cdot$09 & 0$\cdot$07 & 0$\cdot$94 &  0$\cdot$08 & 0$\cdot$08 & 0$\cdot$94 \\

\tiny{$\frac{9}{10}$N(0,1)+$\frac{1}{10}$N(4,1)}  & 0$\cdot$91 & 0$\cdot$72& 0$\cdot$87 & 0$\cdot$95 & 0$\cdot$65 & 0$\cdot$90 & 0$\cdot$94 & 0$\cdot$74 &  0$\cdot$89 & 0$\cdot$94 & 0$\cdot$73 \\
\hline
\end{tabular}
}
\end{center}
\end{table}

It should be noted that the Student's $t$ distributions considered in the study don't satisfy the conditions of Lemma \ref{korrlemma}, rendering $\rho_2$ and $\rho_3$ meaningless. This is discussed further in the results section below.

The simulations were carried out in R, using \texttt{shapiro.test} in the \texttt{stats} package for the Shapiro--Wilk test and \texttt{jarque.bera.test} in the \texttt{tseries} package for the Jarque-Bera test. For Vasicek's test the critical values given in \citet{va1} were used. Critical values for $\sqrt{b_1}$, $b_2$, $|Z_2|$, $Z_2$, $Z_3$, $|Z_2'|$, $Z_2'$ and $Z_3'$ were estimated using 10,000 simulated normal samples for each $n$.

\subsection{Results}
The $Z_2$, $Z_3$, $Z_2'$ and $Z_3'$ tests all performed very well in the study. The results are presented in Tables \ref{tab2}-\ref{tab3} below. There was little difference between the performance of the $Z_2$ and $Z_2'$ tests and between the $Z_3$ and $Z_3'$ tests. The former is interesting, since the $Z_2$ test statistic is an estimator of $\rho(\bar{X},(S^2)^{1/3})$ while $Z_2'$ is an estimator of $\rho(\bar{X},S^2)$.

Judging from the simulation results, we make the recommendations that follow below. Naturally, these are limited to the tests considered in the study. It should however be noted that the Shapiro--Wilk, Vasicek, $\sqrt{b_1}$ and $b_2$ tests have displayed good performance compared to other tests for normality in previous power studies.

For asymmetric alternatives either the $Z_2$ or the $Z_2'$ test should be used. They had the highest power against most asymmetric alternatives in the study, and power close to that of the best test whenever they didn't have the highest power. The one-sided tests are particularly powerful, but the two-sided tests are often more powerful than the competing tests.

For symmetric alternatives either the $Z_3$ or the $Z_3'$ test can be recommended, both for platykurtic ($\kappa<0$) and leptokurtic ($\kappa>0$) distributions. Vasicek's test and the $b_2$ test are more powerful against some alternatives and should be considered to be interesting alternatives to the $Z_3$ or the $Z_3'$ tests. It would be of some interest to compare these tests in a larger power study.

As for the Student's $t$ distributions studied, we note that $\rho_2$ is undefined for the distributions with 4 or fewer degrees of freedom and that $\rho_3$ is undefined for all six distributions. Nevertheless, both tests perform quite well against those alternatives. This is perhaps not unexpected, since the heavy tails of those distributions will cause observations that are so large that they dominate $\sum_i(x_i-\bar{x})^k$ completely. Such observations force $Z_2'$ to be close to either -1 or 1 and $Z_3'$ to be close to 1.

\subsection{Concluding remarks}
In many situations of practical interest the practitioner has some idea about the type of non-normality that can occur; ideas about the sign of the skewness of the alternative and whether or not is has long or short tails. Similarly, it might be of interest to guard against some special class of alternatives. For instance, leptokurtic alternatives with $\kappa>0$ are often considered to be a greater problem than platykurtic alternatives with $\kappa<0$. Judging from the simulation results presented here, the one-tailed $Z_2'$ and $Z_3'$ tests can be recommended above some of the most common tests for normality in such cases.

The good performance of the $Z$ and $Z'$ tests and the fact that the jackknife approach yields tests with essentially the same power as the ''exact'' tests is encouraging. Jackknifing or bootstraping to estimate correlations, or other quantities, could perhaps be used for other independence characterizations as well, as mentioned in \citet{wm1}. \citet{br1} studied some sub- and resampling based tests based on independence characterizations and noted that the bootstrap and jackknife tests seemed to complement each other.

\end{document}